\documentclass[letterpaper,10pt,oneside,onecolumn,final]{amsart}
\usepackage{amsmath,amsfonts,amssymb}
\usepackage[bookmarksnumbered, colorlinks, plainpages]{hyperref}
\hypersetup{colorlinks=true,linkcolor=red, anchorcolor=green, citecolor=cyan, urlcolor=red, filecolor=magenta, pdftoolbar=true}

\newtheorem{theorem}{Theorem}[section]
\newtheorem{lemma}[theorem]{Lemma}
\newtheorem{proposition}[theorem]{Proposition}
\newtheorem{corollary}[theorem]{Corollary}
\theoremstyle{definition}
\newtheorem{definition}[theorem]{Definition}

\theoremstyle{remark}
\newtheorem{remark}[theorem]{Remark}
\numberwithin{equation}{section}
\newcommand{\R}{\mathbb{R}}

\begin{document}
\setcounter{page}{1}

\title[Hermite-Hadamard for Strongly Harmonic Convex Set-Valued Functions]{Some Hermite-Hadamard Inequalities for Strongly Harmonic Convex Set-Valued Functions}

\author[Gabriel Santana, Maira Valera-L\'opez]{Gabriel Santana$^1$, Maira Valera-L\'opez$^{1*}$}
\address{$^{1}$ Escuela de Matem\'atica,  Facultad de Ciencias, Universidad Central de Venezuela, Caracas 1010, Venezuela.}
\email{\textcolor[rgb]{0.00,0.00,0.84}{gaszsantana@gmail.com}}
\email{\textcolor[rgb]{0.00,0.00,0.84}{maira.valera@ciens.ucv.ve}}


\date{Received: xxxxxx; Revised: yyyyyy; Accepted: zzzzzz.
\newline \indent $^{*}$ Corresponding author}

\begin{abstract}
	This research aimed to explore some new Hermite-Hadamard inequalities for strongly harmonic convex set-valued functions with modulus $c>0$ introduced by G. Santana \cite{Santana2018}. 
\end{abstract} \maketitle

\section{Introduction}

The notion of set-valued functions arises in 1963, when C. Berge in \cite{Berge1963}, in the work titled \textit{Topological space: Including a treatment of multi-valued functions, vector spaces and convexity}, introduced the concept of upper and lower limit of set successions and is motivated by its applications in differential and integral analysis, in optimization theory and the calculation of variations, among others (see \cite{Merentes2013}). 

Recently, different researchers have been studied for different notions of set-valued convexity functions as well as have been used to find the error in some inclusion problems with set restrictions (see \cite{Geletu2006,Lara2019,Leiva2013,Matkowski1998,Narvaes2011,Nikodem1987,Nikoden1989,Nikodem2003,Nikodem2014}). One research that give us a important result on this dynamic field is the paper by O. Mejia \textit{et al.} in \cite{Mejia2014} extended the definition of strongly convexity of set-valued maps given by H. Hung (see \cite{Huang20101,Huang20102}) and introduced strong concavity for set-valued functions.

Motivated by this and others research in this mathematics field, and inspired by the work of Noor M. A. \textit{et. al.} \cite{Noor2016} we derive some Hermite-Hadamard inequalities for strongly harmonic set-valued convex functions, introduced by G. Santana in \cite{Santana2018}. 

\section{Preliminary}

First of all we must introduce the classes of harmonic convex set as follow:

\begin{definition}{\cite{Noor2016}}
	Let $X$ be a linear space. A nonempty subset $D$ of $X$ is said harmonic convex, if for all $x,y\in D$ and $t\in [0,1]$, we have:
	
	$$\frac{xy}{tx+(1-t)y}\in D.$$ 
\end{definition}

\noindent In \cite{Iscan2014}, \.{I}. \.{I}scan introduced a new variation of convexity called harmonically convex function and defines it as a real function $f: D\rightarrow \R$ that for all $x,y\in D$ and $t\in [0,1]$,  satisfies the following expression:

\begin{equation}\label{Pre-eq1}
f\left(\frac{xy}{tx+(1-t)y}\right)\leq tf(y)+(1-t)f(x).
\end{equation}

If in (\ref{Pre-eq1}) the sense of inequality is changed, it is obtained that $f$ is harmonically concave.

In the same way, M. A. Noor \textit{et al.} (see \cite{Noor2016}) extends this definition to strongly convexity.

\begin{definition}{\cite{Noor2016}}
	A  function $f:D\subset \R\backslash\{0\} \rightarrow \R$ is  said  to  be  strongly  harmonic convex function with modulus $c >0$, if
	\begin{equation}\label{Pre-eq2}
	f\left(\frac{xy}{tx+(1-t)y}\right)\leq tf(y)+(1-t)f(x)-ct(1-t)\left\|\frac{x-y}{xy}\right\|^{2}.
	\end{equation}
	The function $f$ is said to be strongly harmonic mid-convex function with modulus $c>0$, if (\ref{Pre-eq2}) is assumed only for $t =\frac{1}{2}$, that is:
	$$f\left(\frac{2xy}{x+y}\right) \leq \frac{f(x)+f(y)}{2}-\frac{c}{4}\left\|\frac{x-y}{xy}\right\|^{2}.$$
\end{definition} 

\begin{remark}
	In the following $B$ denote the closed unit ball of $Y$ and $c>0$ a scalar. Also, we consider $X,Y$ linear spaces and we denote $c(Y),cc(Y), n(Y)$ as the convex, compact convex and not empty subsets of $Y$, respectively.
\end{remark}

In this sense, G. Santana \textit{et. al.}  in  \cite{Santana2018} the definition of strongly harmonic convex function and set-valued function are combined, then introduced the strongly harmonic convex set-valued function with modulus $c>0$ as follows:

\begin{definition}{\cite{Santana2018}}
	Let $ X, Y $ be two linear spaces and $ D \subset X $  a harmonically convex set. A set-valued function $ F: D \rightarrow n(Y) $ is called \textit{strongly harmonic convex with modulus $c>0$} if:
	
	\begin{equation}\label{Def1}
	tF(y)+(1-t)F(x) + ct(1-t)\left\|\frac{x-y}{xy}\right\|^{2}B \subset F\left(\frac{xy}{tx+(1-t)y}\right) .
	\end{equation}
	for all $x, y \in D $,  $ t \in [0,1] $ and $B$ a unitary ball of $Y$.
	
	\noindent If we consider $c= 0$ in (\ref{Def1}), $F$ is a  \textit{harmonic convex set-valued function} and considering $t=1/2$ in (\ref{Def1}), them $F$ is a \textit{strongly harmonic midconvex set-valued function with modulus $c>0$}, if: 	
	
	\begin{equation*}
	\frac{F(y)+F(x)}{2} + \frac{c}{4}\left\|\frac{x-y}{xy}\right\|^{2}B \subset F\left(\frac{2xy}{x+y}\right).
	\end{equation*}
\end{definition}

\begin{definition}{\cite{Latif2012}}
	A function $f : [a, b] \subset \R \backslash \{0\} \rightarrow \R$ is said to be harmonic symmetric with respect to $\frac{2ab}{a+b}$, if
	\[f(x) = f \left(\frac{abx}{(a + b)x-ab}\right),\]
	for all $x \in [a, b].$
\end{definition}

Given this definition for functions, we introduce the analog version for set-valued functions in the following definition.

\begin{definition}
	A set-valued function $F : D \subset X \backslash \{0\} \rightarrow n(Y)$ is said to be harmonic symmetric with respect to $\frac{2ab}{a+b}$, if
	\[F(x) = F \left(\frac{abx}{(a + b)x-ab}\right),\]
	holds for all $ x \in D.$
\end{definition}

We introduce the following result which give some relationship between strongly harmonic convex set-valued function with modulus $c>0$ and harmonic convex set-valued function.

\begin{lemma}\begin{itemize}
		\item[i)] A set-valued function $F: D\subset X \rightarrow n(Y)$ is strongly harmonic convex with modulus $c>0$, if and only if, the set-valued function $G(x)= F(x)+c\left\|\frac{1}{x}\right\|^{2}B$ is harmonic convex.
		\item[ii)] A set-valued function $F: D\subset X \rightarrow n(Y)$ is strongly harmonic mid-convex modulus $c>0$, if and only if, the set-valued function $G(x)= F(x)+c\left\|\frac{1}{x}\right\|^{2}B$ is harmonic mid-convex.
	\end{itemize}
\end{lemma}

\begin{proof}
	\begin{description}
		\item[i)] First, let's prove the sufficient condition.	Since $F: D\subset X \rightarrow n(Y)$ is a strongly harmonic set-valued function with modulus $c$, then for all $x,y \in D$, $t\in [0,1]$ and $B$ a unitary ball of $Y$, we have:
		\begin{eqnarray*}
			G\left(\frac{xy}{tx+(1-t)y}\right)&=& F\left(\frac{xy}{tx+(1-t)y}\right)+ c\left\|\frac{tx+(1-t)y}{xy}\right\|^{2}B\\
			&\supset& \left[tF(y)+(1-t)F(x) + ct(1-t)\left\|\frac{x-y}{xy}\right\|^{2}B\right]+ c\left\|\frac{tx+(1-t)y}{xy}\right\|^{2}B\\
			&=& \left[tF(y)+(1-t)F(x) + ct(1-t)\left\|\frac{1}{y}-\frac{1}{x}\right\|^{2}B\right]+ c\left\|\frac{t}{y}+\frac{1-t}{x}\right\|^{2}B\\
			&=&  tF(y)+(1-t)F(x) + ct(1-t)\left[\left\|\frac{1}{y}\right\|^{2}-\frac{2}{xy}+\left\|\frac{1}{x}\right\|^{2}\right]B\\ 
			& & \hspace{2.90cm} + c\left[\left\|\frac{t}{y}\right\|^{2}+\frac{2t(1-t)}{xy}+\left\|\frac{1-t}{x}\right\|^{2}\right]B\\
			&=& tF(y)+(1-t)F(x) + ct(1-t)\left[\left(\frac{1}{1-t}\right)\left\|\frac{1}{y}\right\|^{2}+\left(\frac{1}{t}\right)\left\|\frac{1}{x}\right\|^{2}\right]B\\
			&=& t\left(F(y)+ c\left\|\frac{1}{y}\right\|^{2}B\right)+ (1-t)\left(F(x)+ c\left\|\frac{1}{x}\right\|^{2}B\right)\\
			&=& tG(y)+(1-t)G(x),
		\end{eqnarray*}
		which gives that $G$ is harmonic convex set-valued function.
		
		Moreover, to demonstrate the necessary condition, if $G$ is harmonic convex set-valued function, then we have:
		\begin{eqnarray*}
			F\left(\frac{xy}{tx+(1-t)y}\right)&=& G\left(\frac{xy}{tx+(1-t)y}\right)- c\left\|\frac{tx+(1-t)y}{xy}\right\|^{2}B\\
			&\supset& [tG(y)+(1-t)G(x)]-c\left\|\frac{tx+(1-t)y}{xy}\right\|^{2}B\\
			&=& tG(y)+(1-t)G(x) - c\left[\left\|\frac{t}{y}\right\|^{2}+\frac{2t(1-t)}{xy}+\left\|\frac{1-t}{x}\right\|^{2}\right]B\\
			&=& t\left(G(y)-c\left\|\frac{1}{y}\right\|^{2}B\right)+ (1-t)\left(G(x)-c\left\|\frac{1}{x}\right\|^{2}B\right)\\
			& &+ ct(1-t)\left[\left\|\frac{1}{y}\right\|^{2}-\frac{2}{xy}+\left\|\frac{1}{x}\right\|^{2}\right]B\\
			&=& tF(y)+(1-t)F(x) + ct(1-t)\left\|\frac{x-y}{xy}\right\|^{2}B,
		\end{eqnarray*}
		which show that $F$ is strongly harmonic convex set-valued function with modulus $c>0$.
		
		\item[ii)]  To prove this item, we will start by demonstrate the validity of sufficiency. Since $F$ strongly harmonic mid-convex set-valued function with modulus $c>0$. Then:
		
		\begin{eqnarray*}
			G\left(\frac{2xy}{x+y}\right) &=& F\left(\frac{2xy}{x+y}\right) + c\left\|\frac{x+y}{2xy}\right\|^{2}B\\
			&\supseteq& \left[\frac{F(x)+F(y)}{2}+ \frac{c}{4}\left\|\frac{x-y}{xy}\right\|^{2}B\right]+ \frac{c}{4}\left\|\frac{x+y}{xy}\right\|^{2}B\\
			&=& \frac{F(x)+F(y)}{2} + \frac{c}{4}\left[2\left\|\frac{1}{y}\right\|^{2}+ 2\left\|\frac{1}{x}\right\|^{2}\right]B\\
			&=& \frac{F(x)+F(y)}{2}+ \frac{c}{2}\left\|\frac{1}{y}\right\|^{2}B+ \frac{c}{2}\left\|\frac{1}{x}\right\|^{2}B\\
			&=& \frac{G(x)+G(y)}{2},
		\end{eqnarray*}
		which gives that $G$ is a harmonic mid-convex set-valued function.
		
		 To the necessary implication we will consider that $G$ is harmonic mid-convex set-valued function. Then we have
		\begin{eqnarray*}
			F\left(\frac{2xy}{x+y}\right) &=& G\left(\frac{2xy}{x+y}\right) - c\left\|\frac{x+y}{2xy}\right\|^{2}B\\
			&\supseteq&  \frac{G(x)+G(y)}{2} -\frac{c}{4}\left[\left\|\frac{1}{y}\right\|^{2}+\frac{2}{xy}+\left\|\frac{1}{x}\right\|^{2}\right]B \\
			&=& \frac{G(x)-c\left\|\frac{1}{x}\right\|^{2}B}{2}+  \frac{G(y)-c\left\|\frac{1}{y}\right\|^{2}B}{2}+\frac{c}{4}\left[\left\|\frac{1}{y}\right\|^{2}-\frac{2}{xy}+\left\|\frac{1}{x}\right\|^{2}\right]B \\
			&=& \frac{F(x)+F(y)}{2}+ \frac{c}{4}\left\|\frac{x-y}{xy}\right\|^{2}B,
		\end{eqnarray*}
		which shows that $F$ is a strongly harmonic mid-convex set-valued function with modulus $c>0$.
 		
	\end{description}

\end{proof}	
	
The following theorem is used in this research for our results, this represent a Hermite-Hadamard inequalities for strongly convex set valued function, introduced by Nikoden N. \textit{et. al} in \cite{Nikodem2014}.

\begin{theorem}\cite{Nikodem2014}
	If a set-valued function $F: I\rightarrow cl(Y)$ is strongly convex with modulus $c>0$, then
	\begin{equation}\label{Pre-eq3}
	\frac{1}{b-a}\int_{a}^{b}F(x)dx + \frac{c}{12}(b-a)^{2}B \subset F\left(\frac{a+b}{2}\right),
	\end{equation}
	and 
	\begin{equation}\label{Pre-eq4}
	\frac{F(a)+F(b)}{2}+ \frac{c}{6}(b-a)^{2}B\subset \frac{1}{b-a}\int_{a}^{b}F(x)dx,
	\end{equation}
\end{theorem} 
 for all $a,b \in I$, $a<b$ and $B$ is a unitary ball of $Y$.

\section{Main Results}

To start with the main results of this work which is to obtain some Hermite-Hadamard inequalities for functions, we need the following proposition. 

\begin{proposition}\label{prop_3_1}
	Let $D= [a,b] \subseteq X\backslash \{0\}$ and consider the function $G: \left[\frac{1}{b},\frac{1}{a}\right] \rightarrow n(Y)$ defined by $G(t) = F\left(\frac{1}{t}\right)$. Then $F$ is strongly harmonic convex set-valued function on $[a,b]$ with modulus $c>0$, if and only if, $G$ is strongly convex set-valued function on $\left[\frac{1}{b},\frac{1}{a}\right]$ with modulus $c>0$.
\end{proposition}

\begin{proof}
	 First, let's prove the sufficient condition.	
	 
	 Let $F$ be a strongly harmonic convex set-valued function with modulus $c> 0$. Then
	$$tF(y)+(1-t)F(x)+ct(1-t)\left\|\frac{x-y}{xy}\right\|^2 B\subseteq F\left(\frac{xy}{tx+(1-t)y}\right)$$
	for all $x,y\in [a,b]$, $t\in [0,1]$ and $c>0$.\\
	
	Now, consider $G(t)=F\left(\frac{1}{t}\right)$ and suppose that $x= \frac{1}{x_{2}}$ and $y= \frac{1}{x_{1}}$ in $[a,b]$. First, from the left side of the previous expression we have:
	
	\begin{eqnarray*}
		& & tF(y)+(1-t)F(x)+ct(1-t)\left\|\frac{x-y}{xy}\right\|^2 B\\
		&=& tF\left(\frac{1}{x_{1}}\right)+(1-t)F\left(\frac{1}{x_{2}}\right)+ct(1-t)\left\|\frac{\frac{1}{x_{2}}-\frac{1}{x_{1}}}{\frac{1}{x_{2}}\frac{1}{x_{1}}}\right\|^2 B\\
		&=& tG(x_{1})+ (1-t)G(x_{2})+ ct(1-t)\|x_{2}-x_{1}\|^2 B,
	\end{eqnarray*}   
     while on the right side we get:
	
	\begin{eqnarray*}
		F\left(\frac{\frac{1}{x_{2}}\frac{1}{x_{1}}}{t\frac{1}{x_{2}}+(1-t)\frac{1}{x_{1}}}\right) &=& F\left(\frac{1}{tx_{1}+(1-t)x_{2}}\right)\\
		&=& G(tx_{1}+(1-t)x_{2}).
	\end{eqnarray*}
	Accordingly,
	$$tG(x_{1})+ (1-t)G(x_{2})+ ct(1-t)\|x_{2}-x_{1}\|^2 B \subseteq G(tx_{1}+(1-t)x_{2})$$
	for all $x_{1},x_{2}\in \left[\frac{1}{b},\frac{1}{a}\right]$, $t\in [0,1]$ and $c> 0$.
	
	Thus $G$ is strongly convex set-valued function modulus $c$.
	
	Now, let's try the necessary condition.
	
	Let $G$ strongly convex set-valued function with modulus $c>0$, then:
	
	$$tG(x_{1})+ (1-t)G(x_{2})+ ct(1-t)\|x_{2}-x_{1}\|^2 B \subseteq G(tx_{1}+(1-t)x_{2}),$$	
	for all $x_{1},x_{2}\in \left[\frac{1}{b},\frac{1}{a}\right]$, $t\in [0,1]$ and $c> 0$.
	
	Whereas $G(t)=F\left(\frac{1}{t}\right)$, $x=\frac{1}{x_{2}}$ and $y=\frac{1}{x_{1}}$ in $\left[\frac{1}{b},\frac{1}{a}\right]$, from the left side of the previous expression we have:
	
	\begin{eqnarray*}
		& & tG\left(\frac{1}{x_{2}}\right)+(1-t)G\left(\frac{1}{x_{1}}\right)+ct(1-t)\left\|\frac{1}{x_{1}}-\frac{1}{x_{2}}\right\|^2 B\\
		&=& tF(x_{2})+ (1-t)F(x_{1})+ct(1-t)\left\|\frac{x_{2}-x_{1}}{x_{2}x_{1}}\right\|^2 B.
	\end{eqnarray*}  
	On the right side of contention we have:
	
	\begin{eqnarray*}
		G\left(t\frac{1}{x_{2}}+(1-t)\frac{1}{x_{1}}\right) &=& G\left(\frac{tx_{1}+(1-t)x_{2}}{x_{2}x_{1}}\right)\\
		&=& F\left(\frac{x_{2}x_{1}}{tx_{1}+(1-t)x_{2}}\right).
	\end{eqnarray*}   
	Then
	$$tF(x_{2})+ (1-t)F(x_{1})+ct(1-t)\left\|\frac{x_{2}-x_{1}}{x_{2}x_{1}}\right\|^2 B \subseteq F\left(\frac{x_{2}x_{1}}{tx_{1}+(1-t)x_{2}}\right),$$ 
	for all $x_{1},x_{2} \in [a,b]$, $t\in [0,1]$ and $c>0$.
	Thus $F$ is strongly harmonic convex set-valued function with modulus $c$.
	
\end{proof}

Now, using the Proposition \ref{prop_3_1}, we obtain a new result of Hermite-Hadamard inequality type for strongly harmonic convex set-valued functions. 

\begin{theorem}
	Let $D$ be a harmonic set of $X$ and $F: D\subset X\rightarrow n(Y)$ strongly harmonic convex set-valued function modulus $c>0$. If $F$ is Aumann integrable (see \cite{Aumann1965}), then:
	\begin{equation*}
	 \frac{ab}{b-a}\int_{a}^{b}\frac{F(x)}{x^2}dx + \frac{c}{12}\left\|\frac{b-a}{ab}\right\|^{2}B \subset F\left(\frac{2ab}{a+b}\right),
	\end{equation*} 
	and
	\begin{equation*}
	\frac{F(a)+F(b)}{2}+ \frac{c}{6}\left\|\frac{b-a}{ab}\right\|^{2}B\subset \frac{ab}{b-a}\int_{a}^{b}\frac{F(x)}{x^2}dx,
	\end{equation*}
	for all $a,b \in D$, $t\in [0,1]$ and $B$ a unitary ball of $Y$. 
\end{theorem}

\begin{proof}
	Let $G(t)= F\left(\frac{1}{t}\right)$ be  a strongly convex set-valued function modulus $c>0$, then $G$ satisfies (\ref{Pre-eq3}) and (\ref{Pre-eq4}) for all $x,y \in D$, i.e.:
	
	\begin{equation*}
		\frac{1}{y-x}\int_{x}^{y}G(t)dt + \frac{c}{12}(y-x)^{2}B \subset G\left(\frac{x+y}{2}\right),
	\end{equation*}
	and 
	\begin{equation*}
		\frac{G(x)+G(y)}{2}+ \frac{c}{6}(y-x)^{2}B\subset \frac{1}{y-x}\int_{x}^{y}G(t)dt,
	\end{equation*}
	for all $t \in [x,y]$. Now consider $G(t)= F\left(\frac{1}{t}\right)$ with $t\in \left[\frac{1}{b},\frac{1}{a}\right]$, then on the one hand, we take:  
	
	\begin{equation*}
		\frac{1}{\frac{1}{b}-\frac{1}{a}}\int_{\frac{1}{b}}^{\frac{1}{a}}F\left(\frac{1}{t}\right)dt + \frac{c}{12}\left(\frac{1}{a}-\frac{1}{b}\right)^{2}B \subset F\left(\frac{1}{\frac{\frac{1}{b}+\frac{1}{a}}{2}}\right).
	\end{equation*}  
	Thus,
    
    \begin{equation*}
    	\frac{ab}{b-a}\int_{\frac{1}{b}}^{\frac{1}{a}}F\left(\frac{1}{t}\right)dt + \frac{c}{12}\left\|\frac{b-a}{ab}\right\|^{2}B \subset F\left(\frac{2ab}{a+b}\right).
    \end{equation*}
    Then, by Aumann's definition, we take the following result:
	
	\begin{eqnarray*}
		\int_{\frac{1}{b}}^{\frac{1}{a}}F\left(\frac{1}{t}\right)dt &=& \left\{\int_{\frac{1}{b}}^{\frac{1}{a}}f\left(\frac{1}{t}\right)dt: f\in F\right\}\\
		&=& \left\{\int_{a}^{b}\frac{f(x)}{x^2}dx: f\in F\right\}\\
		&=& \int_{a}^{b}\frac{F(x)}{x^2}dx.
	\end{eqnarray*}
   	Therefore,
   	
   	\[\frac{ab}{b-a}\int_{a}^{b}\frac{F(x)}{x^2}dx + \frac{c}{12}\left\|\frac{b-a}{ab}\right\|^{2}B \subset F\left(\frac{2ab}{a+b}\right).\]
   	On the other hand, we have that:
   	
   	\[\frac{F\left(\frac{1}{\frac{1}{b}}\right)+ F\left(\frac{1}{\frac{1}{a}}\right)}{2}+ \frac{c}{6}\left\|\frac{b-a}{ab}\right\|^{2}B\subset \frac{ab}{b-a}\int_{a}^{b}\frac{F(x)}{x^2}dx,\]
   and this is,    
  
   	\[\frac{F(a)+F(b)}{2}+ \frac{c}{6}\left\|\frac{b-a}{ab}\right\|^{2}B\subset \frac{ab}{b-a}\int_{a}^{b}\frac{F(x)}{x^2}dx.\]
\end{proof}

\begin{theorem}\label{mr-th2}
	Let $F,G: D\rightarrow n(Y)$ be strongly harmonic convex set-valued function with modulus $c>0$. If $F,G$ are Aumann integrable, then:
	\begin{equation}\label{mr-eq1}
	\frac{1}{6}M(a,b)+\frac{1}{3}N(a,b)+S(a,b) \frac{c}{12}\left\|\frac{b-a}{ab}\right\|^{2}B + \frac{c^2}{30}\left\|\frac{b-a}{ab}\right\|^{4}B\subset \frac{ab}{b-a}\int_{a}^{b}\frac{F(x)G\left(\frac{abx}{(a+b)x-ab}\right)}{x^2}dx,
	\end{equation}
	where
	\[M(a,b)= F(a)G(a)+F(b)G(b),\]
	\[N(a,b)= F(a)G(b)+F(b)G(a),\]
	\[S(a,b)= F(a)+F(b)+G(a)+G(b),\]
\end{theorem}

\begin{proof}
	Let $F,G$ strongly harmonic convex set-valued functions with modulus $c>0$. Then:
	\begin{eqnarray*}
		\frac{ab}{b-a}\int_{a}^{b}\frac{F(x)G\left(\frac{abx}{(a+b)x-ab}\right)}{x^2}dx	&=& \int_{0}^{1}F\left(\frac{ab}{ta+(1-t)b}\right)G\left(\frac{ab}{(1-t)a+tb}\right)dt\\
		&\supseteq& \int_{0}^{1}\left[tF(b)+ (1-t)F(a)+ ct(1-t)\left\|\frac{b-a}{ab}\right\|^{2}B\right]\\
		&&\hspace{3cm}\left[tG(a)+ (1-t)G(b)+ ct(1-t)\left\|\frac{b-a}{ab}\right\|^{2}B\right]dt\\
		&=& F(a)G(a)\int_{0}^{1}(1-t)^{2}dt+ F(b)G(a)\int_{0}^{1} t^{2}dt\\
		& &\hspace{3cm} + [F(a)G(a)+F(b)G(b)]\int_{0}^{1}t(1-t)dt \\
		& &\hspace{3cm} + c\left\|\frac{b-a}{ab}\right\|^{2}B[F(a)+G(b)]\int_{0}^{1}t(1-t)^{2}dt\\
		& &\hspace{3cm} + c\left\|\frac{b-a}{ab}\right\|^{2}B [F(b)+G(a)] \int_{0}^{1} t^{2}(1-t)dt \\
		& &\hspace{3cm} + c^{2}\left\|\frac{b-a}{ab}\right\|^{4}B\int_{0}^{1}t^{2}(1-t)^{2}dt\\
		&=& \frac{F(a)G(b)+F(b)G(a)}{3}+ \frac{F(a)G(a)+F(b)G(b)}{6}\\
		& & +[F(a)+F(b)+G(a)+G(b)] \frac{c}{12}\left\|\frac{b-a}{ab}\right\|^{2}B+ \frac{c^2}{30}\left\|\frac{b-a}{ab}\right\|^{4}B\\
		&=& \frac{1}{6}M(a,b)+\frac{1}{3}N(a,b)+S(a,b) \frac{c}{12}\left\|\frac{b-a}{ab}\right\|^{2}B + \frac{c^2}{30}\left\|\frac{b-a}{ab}\right\|^{4}B,
	\end{eqnarray*}
which is the required result.
\end{proof}

\noindent If $F=G$ in theorem \ref{mr-eq1}, then it reduces to the following result.

\begin{corollary}
	Let $F: D\subset X \rightarrow n(Y)$ be strongly harmonic convex set-valued function with modulus $c>0$. If $F$ is Aumann integrable on $[a,b]$, then:
	\begin{equation*}
	\frac{2F(a)F(b)}{3}+ \frac{F^2(a)+F^2(b)}{6}+ \frac{c}{6}\left\|\frac{b-a}{ab}\right\|^{2}[F(a)+F(b)]+\frac{c^2}{30}\left\|\frac{b-a}{ab}\right\|^{4} \subseteq \frac{ab}{b-a}\int_{a}^{b}\frac{F(x)F\left(\frac{abx}{(a+b)x-ab}\right)}{x^2}dx.
	\end{equation*}
\end{corollary}

\begin{theorem}
	Let $F,G: D\subset X \rightarrow n(Y)$ be two strongly harmonic convex set-valued functions with modulus $c>0$. If $F,G$ are Aumann integrable on $[a,b]$, then:
	\[\frac{1}{6}M(a,b)+\frac{1}{3}N(a,b)+S(a,b) \frac{c}{12}\left\|\frac{b-a}{ab}\right\|^{2}B + \frac{c^2}{30}\left\|\frac{b-a}{ab}\right\|^{4}B\subset \frac{ab}{b-a}\int_{a}^{b}\frac{F(x)G(x)}{x^2}dx,\]
	where $M(a,b), N(a,b)$ and $S(a,b)$ are given by theorem \ref{mr-th2}.
\end{theorem}

\begin{proof}
	Let $F,G$ be two strongly harmonic convex set-valued functions with modulus $c>0$. Then:
	\begin{eqnarray*}
		&& \frac{ab}{b-a}\int_{a}^{b}\frac{F(x)G(x)}{x^2}dx\\ 
		&=& \int_{0}^{1}F\left(\frac{ab}{ta+(1-t)b}\right)G\left(\frac{ab}{ta+(1-t)b}\right)dt\\
		&\supseteq&  \int_{0}^{1} \left[tF(b)+(1-t)F(a)+ct(1-t)\left\|\frac{b-a}{ab}\right\|^{2}B\right]\left[tG(b)+(1-t)G(a)+ct(1-t)\left\|\frac{b-a}{ab}\right\|^{2}B\right]dt\\
		&=& F(a)G(a)\int_{0}^{1}(1-t)^{2}dt+ [F(a)G(b)+F(b)G(a)]\int_{0}^{1}t(1-t)dt\\ 
		&& + F(b)G(b)\int_{0}^{1}t^{2}dt+ c\left\|\frac{b-a}{ab}\right\|^{2}B[F(a)+G(a)]\int_{0}^{1}t(1-t)^{2}dt\\
		&& + c\left\|\frac{b-a}{ab}\right\|^{2}B[F(b)+G(b)]\int_{0}^{1}t^{2}(1-t)dt + c^{2}\left\|\frac{b-a}{ab}\right\|^{4}B\int_{0}^{1}t^{2}(1-t)^{2}dt\\
		&=& \frac{1}{3}[F(a)G(a)+F(b)G(b)]+\frac{1}{6}[F(a)G(b)+F(b)G(a)] +\frac{c}{12}[F(a)+F(b)+G(a)+G(b)]\left\|\frac{b-a}{ab}\right\|^{2}\\
		&& +\frac{c^{2}}{30}\left\|\frac{b-a}{ab}\right\|^{4}B\\
		&=& \frac{1}{6}M(a,b)+\frac{1}{3}N(a,b)+S(a,b) \frac{c}{12}\left\|\frac{b-a}{ab}\right\|^{2}B + \frac{c^2}{30}\left\|\frac{b-a}{ab}\right\|^{4}B,
	\end{eqnarray*} 
which is the required result.
\end{proof}	

If $F=G$ in previously Theorem, then it reduces to the following result.

\begin{corollary}
	Let $F:D\rightarrow n(y)$ strongly harmonic convex set-valued function with modulus $c>0$. If $F$ is Aumann integrable, then:
	\[\frac{F^2(a)+F^2(b)+F(a)+F(b)}{3}+\frac{c}{6}\left\|\frac{b-a}{ab}\right\|^{2}B[F(a)+F(b)]+\frac{c^2}{30}\left\|\frac{b-a}{ab}\right\|^{4}B \subset \frac{ab}{b-a}\int_{a}^{b}\frac{F^2(x)}{x^2}dx.\]
\end{corollary}


\begin{thebibliography}{99}
	
	\bibitem{Aumann1965} Aumann, R. J. (1965). \emph{Integrals of set-valued functions}. Journal of mathematical analysis and applications, \textbf{12}(1), 1-12.
	
	\bibitem{Berge1963} Berge, C. (1963). \emph{Topological Space: Including a treatment of Multi-Valued functions, vector spaces and convexity}, Dover Publications, INC. Mineola, New York.
				
	\bibitem{Geletu2006} Geletu, A. (2006). \emph{Introduction to topological spaces and set-valued maps (Lecture notes)}, Department of Operations Research \& Stochastics Ilmenau University of Technology. August 25.
	
	\bibitem{Huang20101} Huang, H., (2010). \emph{Global error bounds with exponents for multifunctions with set constraints}, Commun. Contemp. Math. 12, 417–435.
	
	\bibitem{Huang20102}  Huang, H., (2010). \emph{Inversion theorem for nonconvex multifunctions}, Math. In-equal. Appl. 13, 841–849.
		
	\bibitem{Iscan2014} \.{I}\c{s}can, \.{I}., (2014), Hermite-Hadamard type inequalities for harmonically convex functions. Hacettepe  Journal of Mathematics and Statistics, 43(2). 935-942.
	
	\bibitem{Lara2019} Lara . T. , Merentes N., Quintero R. and  Rosales E. (2019). \emph{On $m$-convexity of set-valued functions}, Advance in Operator Theory, \textbf{4} (14), 767-783.
	
	\bibitem{Latif2012} Latif M. A., Dragomir S.S. and Momoniat E. (2012). \emph{Some Fejer type inequalities for harmonically convex functions with applications to special means}, http://rgmia.org/papers/v18/v18a24.pdf.
	
	\bibitem{Leiva2013} Leiva, H., Merentes, N., Nikodem, K., and S\'anchez, J. L. (2013). \emph{Strongly convex set-valued maps}. Journal of Global Optimization, 57(3), 695-705.
		
	\bibitem{Matkowski1998} Matkowski J. and  Nikodem K. (1998), \emph{Convex set-valued functions on $(0,+1)$ and their conjugate}, Rocznik Nauk.-Dydakt. Prace Mat. No. 15 , 103-107.
	
	\bibitem{Mejia2014} Mejia, O., Merentes, N., and Nikodem, K. (2014). \emph{Strongly concave set-valued maps}. Mathematica Aeterna, 4, 477-487.
		
	\bibitem{Merentes2013} Merentes N., Ribas S. (2013).\emph{El desarrollo del concepto de funci\'on convexa}. Ediciones IVIC, Caracas-Venezuela.
	
	\bibitem{Narvaes2011} Narváez D. X., Restrepo G. (2011).\emph{Funciones Multivaluadas}, Facultad de Ciencias Naturales y Exactas Universidad del Valle, Revista de Ciencias, septiembre 20.
		
	\bibitem{Nikodem1987} Nikodem K. (1987), \emph{On concave and midpoint concave set-valued functions}, Glasnik Mat. Ser. III \textbf{22}(42), no. 1, 69-76.
		
	\bibitem{Nikoden1989} Nikodem, K. (1989). \emph{A characterization of midconvex set-valued functions}. Acta Universitatis Carolinae. Mathematica et Physica, 30(2), 125-129.
	
	\bibitem{Nikodem2003} Nikodem K. (2003).\emph{K-convex and K-concave Set-Valued Functions}, Zeszyty Nauk. Politech. L\'odz. Math. 4(3), Art. 52.	
    
    \bibitem{Nikodem2014} Nikodem K., S\'anchez J. L.,S\'anchez L. (2014).\emph{Jensen and Hermite-Hadamard inequalities for strongly convex set-valued maps}, Mathematica Aeterna, Vol. 4, no. 8, 979-987.
    	
	\bibitem{Noor2016} Noor M. A., Noor K. I., Iftikhar S. (2016). \emph{Hermite-Hadamard Inequalities for Strongly Harmonic convex function}. Journal of Inequalities and Special Functions. 7(3), 99-113
				
	\bibitem{Santana2018} Santana G.,  Gonz\'alez L. and  Merentes N. (2019) \emph{Funciones conjunto valuadas arm\'onicamente convexas}, Divulgaciones Matem\'aticas, \textbf{19}(1), 20-33.
		
\end{thebibliography}
 \end{document}